\documentclass[10pt,reqno]{amsart}
\usepackage{amssymb,mathrsfs,mathabx}
\usepackage[utf8]{inputenc}
\usepackage{enumerate}
\usepackage[T1]{fontenc}
\usepackage{lmodern}
\usepackage{color}

\def\bN {\mathbb{N}}

\def\bR {\mathbb{R}}

\def\e{\mathrm{e}}

\def\cA {\mathcal{A}}

\def\cF {\mathcal{F}}

\def\cL {\mathcal{L}}

\def\cS {\mathcal{S}}

\def\cX {\mathcal{X}}

\def\d {{\mathrm d}}

\font\dsrom=dsrom10 scaled 1200
\def \indic{\textrm{\dsrom{1}}}

\newcommand{\Sign}{\operatorname{sgn}}

\newcommand{\Supp}{\operatorname{supp}}

\newcommand{\dis}{\displaystyle}

\newcommand{\ba}{\begin{aligned}}
\newcommand{\ea}{\end{aligned}}

\newcommand{\be}{\begin{equation}}
\newcommand{\ee}{\end{equation}}

\newtheorem{theorem}{Theorem}[section]
\newtheorem{corollary}[theorem]{Corollary}

\newtheorem{lemma}[theorem]{Lemma}
\newtheorem{proposition}[theorem]{Proposition}

\theoremstyle{definition}
\newtheorem{definition}[theorem]{Definition}
\newtheorem{remark}[theorem]{Remark}

\newcommand{\ep}{\varepsilon}

\begin{document}

\title[Growth-fragmentation with bounded fragmentation rate]{Asymptotic behavior of the growth-fragmentation equation\\
with bounded fragmentation rate}

\author[\'E. Bernard]{\'{E}tienne Bernard}
\address[\'E. Bernard]{Laboratoire de G\'eod\'esie, IGN-LAREG, B\^atiment Lamarck A et B, 35 rue H\'el\`ene Brion, 75013 Paris, France.}
\email{etienne.bernard@ign.fr}
\author[P. Gabriel]{Pierre Gabriel}
\address[P. Gabriel]{Laboratoire de Math\'ematiques de Versailles, UVSQ, CNRS, Universit\'e Paris-Saclay,  45 Avenue des \'Etats-Unis, 78035 Versailles cedex, France.}
\email[Corresponding author]{pierre.gabriel@uvsq.fr}


\begin{abstract}
We are interested in the large time behavior of the solutions to the growth-fragmentation equation.
We work in the space of integrable functions weighted with the principal dual eigenfunction of the growth-fragmentation operator.
This space is the largest one in which we can expect convergence to the steady size distribution.
Although this convergence is known to occur under fairly general conditions on the coefficients of the equation, we prove that it does not happen uniformly with respect to the initial data when the fragmentation rate in bounded.
First we get the result for fragmentation kernels which do not form arbitrarily small fragments by taking advantage of the Dyson-Phillips series.
Then we extend it to general kernels by using the notion of quasi-compactness and the fact that it is a topological invariant.
\end{abstract}



\maketitle

\section*{Introduction}

In this article, we study the asymptotic behavior of the {\it growth-fragmentation equation} 
\begin{equation}
\label{eq:GF}
\left\{
    \begin{array}{ll}
    \partial_{t}f(t,x)+\partial_x\left(\tau(x) f(t,x)\right) = \cF f(t,x),\quad & t,x>0,
    \vspace{2mm}\\
    (\tau f)(t,0)=0,&t>0,
    \vspace{2mm}\\
    f(0,x) = f^{\rm{in}}(x),&x\geq0.
    \end{array}
\right.
\end{equation}
This equation appears in the modeling of various physical or biological phenomena~\cite{MD86,BP,Banasiak,Gabriel15}
as well as in telecommunication~\cite{BCGMZ}.
The unknown $f(t,x)$ represents the concentration at time $t$ of some ``particles'' with ``size'' $x>0,$
which can be for instance the volume of a cell, the length of a fibrillar polymer, or the window size in data transmission over the Internet.
Each particle grows with a rate $\tau(x)$ and splits according to the fragmentation operator $\cF$ which acts on a function $f(x)$ through
\[\cF f(x):=\cF_+f(x)-B(x)f(x).\]
The positive part $\cF_+$ is an integral operator given by
\be
\label{def:cF+}
\cF_+ f(x) := \int_0^1B\Bigl(\frac xz\Bigr)f\Bigl(\frac xz\Bigr)\frac{\wp(\d z)}{z}.
\ee
When a particle of size $x$ breaks with rate $B(x),$ it produces smaller particles of sizes $zx$ with $0<z<1$ distributed with respect to the fragmentation kernel $\wp.$
The fragmentation kernel $\wp$ is a finite positive measure on the open interval $(0,1)$ which satisfies
\be\label{as:mass_cons}\int_0^1 z\,\wp(\d z)=1.\ee
This is a mass conservation condition since it ensures that if we sum the sizes of the offsprings we recover the size of the mother particle.

Classical examples of fragmentation kernels are the mitosis kernel $\wp=2\delta_{1/2},$ the asymmetrical division kernels $\wp=\delta_\nu+\delta_{1-\nu}$ with $\nu\in(0,1/2),$
and the power law kernels $\wp(\d z)=(\nu+2)z^\nu\d z$ with $\nu>-2.$
Notice that the power law kernels are physically relevant only for $\nu\leq0$ (see {\it e.g.} discussion in Section~8.2.1 of~\cite{Banasiak}),
which includes the uniform kernel $\wp\equiv2.$

\

The long time behavior of the solutions is strongly related to the existence of $(\lambda,G,\phi)$ solution to the following Perron eigenvalue problem:
\be\label{eq:direct_Perron}(\tau G)'+\lambda G=\mathcal F G,\qquad G\geq0,\qquad \int_0^\infty G(x)\,\d x=1\ee
and the dual problem:
\be\label{eq:dual_Perron}-\tau\phi'+\lambda\phi+B\phi=\mathcal F_+^*\phi,\qquad\phi\geq0,\qquad \int_0^\infty G(x)\phi(x)\,\d x=1\ee
where
\[\cF_+^*\varphi(x):=B(x)\int_0^1\varphi(zx)\wp(\d z).\]
When $(\lambda,G,\phi)$ exists and for initial distributions which satisfy
\begin{equation}
\label{def:crochets}
\langle f^{\rm in},\phi\rangle:=\int_0^\infty f^{\rm in}(y)\phi(y)\,\d y<+\infty,
\end{equation}
the solutions to Equation~\eqref{eq:GF} are expected to behave like
\[f(t,x)\sim\langle f^{\rm in},\phi\rangle\,G(x){\rm e}^{\lambda t}\qquad\text{when}\ t\to+\infty.\]
This property is sometimes called {\it asynchronous exponential growth} since it ensures that the shape of the initial distribution is forgotten for large times.
Asymptotically the population grows exponentially fast with a {\it Malthus parameter} $\lambda$ and is aligned to the {\it stable size distribution} $G.$

\

The question of the convergence to the stable size distribution for growth-fragmentation dates back to~\cite{DHT}.
In this pioneer paper, Diekmann, Heijmans and Thieme consider the case when $x$ lies in a bounded domain,
which is made possible by imposing that the fragmentation rate $B$ blows up in a non-integrable way at a finite maximal size,
in order to prevent the particles to reach the maximal size.
In this case and for the mitosis kernel, they prove that the {\it rescaled} solution $f(t,x)\e^{-\lambda t}$ converges to the stable size distribution (times a constant) for a weighted $L^\infty$-norm.
They use semigroup techniques and the spectral results obtained in~\cite{H85} for the mitosis kernel (and extended to smooth self-similar kernels in~\cite{H84}).
The result in~\cite{DHT} has also been obtained in a $L^1$ setting: in~\cite{GreinerNagel} for the mitosis kernel, with an exponential rate of convergence,
and in~\cite{RP,BPR} for general kernels.

\

When the domain is $\bR_+$, it is necessary to control what goes to infinity.
If we want to have convergence to the stable size distribution, an obvious condition is that $\langle f^{\rm in},\phi\rangle$ is finite.
In particular we cannot expect a convergence in $L^1(\bR_+)$ when $\phi$ is not bounded.
The largest space in which we can work is then
\[L^1_\phi:=L^1(\bR_+,\phi(x)\,\d x).\]
The rate of convergence of the solutions to Equation~\eqref{eq:GF} in this critical functional space is the purpose of the present paper.
It is known from~\cite{MMP2} that for any initial data in $L^1_\phi$ the solution converges to the stable size distribution in the $L^1_\phi$-norm.
A further question concerns the existence of an exponential rate of convergence in $L^1_\phi.$
It was first addressed in~\cite{PR05} for the mitosis kernel and a (almost) constant fragmentation rate $B.$
It is proved that the exponential convergence occurs provided that the initial data is bounded for a stronger norm than the $L^1_\phi$-norm.
There is an additional term which corresponds to the Wasserstein distance between the initial data and the equilibrium (see~\cite{BCGMZ} for the statement with this distance).
In~\cite{LP09} where the result is extended to more general kernels (see also~\cite{PPS} for another generalization), the necessity of this term is questioned.
Can we control the decay rate in the $L^1_\phi$-norm by the same norm on the initial data?
We know that it is possible in smaller spaces.
It has been proved in Hilbert spaces by using  functional inequalities~\cite{CCM11,GS14,Monmarche}.
For weighted $L^1$ spaces we also have positive results~\cite{CCM10,MS} but for weights which are stronger than $\phi.$
The question we address in this paper is then to know whether these weights can be replaced by $\phi.$
We prove that it is not possible when the fragmentation rate $B$ is bounded,
and this ensures a kind of optimality for the previous results.

\

%
%

\section{Statement of the main result}

We start by saying a few words about the moments of $\wp$ which will play a crucial role in the study of the problem.
For any $r\in\bR$ the (possibly infinite) $r$-th moment of $\wp$ is denoted by
\[\wp_r:=\int_0^1z^r\,\wp(\d z).\]
The zero-moment $\wp_0$ represents the mean number of fragments produced by the fragmentation of one particle.
The fact that the measure $\wp$ is finite means that $\wp_0<+\infty.$
Of course $\wp_r$ can become infinite for negative $r.$
Define
\[\underline r:=\inf\{r\in\bR,\ \wp_r<+\infty\}\in[-\infty,0].\]
Clearly $\wp_r<+\infty$ for any $r>\underline r.$
Additionally since $\wp$ is a positive measure on the open interval $(0,1),$ the function $r\mapsto\wp_r$  is strictly decreasing on $(\underline r,+\infty).$
The mass conservation requires $\wp_1=1$ and because $\wp_0>\wp_1$ we deduce that the mean number of fragments is larger than one.

\

We are now ready to state the

\smallskip

\noindent{\bf Hypotheses on the coefficients:}

\medskip

\begin{itemize}

\item[\bf (H$\tau$)] The growth rate $\tau$ is a positive $C^1$-function on $\bR_+^*$ which satisfies
\begin{equation}\label{as:tau0}
\frac1\tau\in L^1_{loc}(\bR_+),
\end{equation}
\begin{equation}\label{as:tau_infty}
\exists\, \underline\alpha\leq\alpha<1,\quad \tau(x)=O(x^{\alpha})\quad\text{and}\quad x^{\underline\alpha}=O\big(\tau(x)\big)\quad\text{when}\ x\to+\infty.
\end{equation}

\medskip

\item[\bf (H$B$)]
The total fragmentation rate $B$ is a nonnegative essentially bounded function on $\bR_+,$ with a connected support, and such that
\begin{equation}\label{as:B_infty}
\exists A_0,B_\infty>0,\quad \forall x\geq A_0,\quad B(x)=B_\infty.
\end{equation}

\medskip

\item[\bf (H$\wp$)]
The fragmentation kernel $\wp$ is a finite positive Borel measure on the open interval $(0,1)$ which satisfies the mass conservation condition~\eqref{as:mass_cons}
and is such that
\begin{equation}\label{as:wpr_lim}
\lim_{r\to\underline r^+}\wp_r=+\infty.
\end{equation}

\end{itemize}

The monotone convergence theorem together with condition~\eqref{as:wpr_lim} ensure that $r\mapsto\wp_r$ is a continuous function on $(\underline r,+\infty)$ and that its image is $(0,+\infty).$
This property will be fundamental in our study of the dual eigenfunction $\phi.$
In order to illustrate which pathological kernels we want to avoid with condition~\eqref{as:wpr_lim}, we give some examples.

\

\noindent{\bf Examples and counter-example}
\begin{enumerate}[i)]
\item Consider the measure with Lebesgue density defined by
\[\wp(z)=\Bigl(\int_0^1|\log z|^{-2}\,\d z\Bigr)^{-1}z^{-1}|\log z|^{-2}.\]
It satisfies the mass conservation condition~\eqref{as:mass_cons} but not Assumption~\eqref{as:wpr_lim}.
In this case $\underline r=0$ and $\wp_0<+\infty.$
The image of $r\mapsto\wp_r$ is $(0,\wp_0].$
\item Consider a kernel which is absolutely continuous with respect to the Lebesgue measure close to the origin and has fast decrease when $z\to0,$ {\it i.e.} such that
\[\forall \mu\in\bR,\quad\wp(z)=o(z^\mu)\quad \text{when}\ z\to0.\]
In this case Assumption~\eqref{as:wpr_lim} is satisfied since $\underline r=-\infty$ and for any $r\leq0$ we have
\[\wp_r=\int_0^1z^r\wp(\d z)\geq -r\int_0^1(1-z)\wp(\d z)=-r(\wp_0-1)\xrightarrow[r\to-\infty]{}+\infty.\]
A concrete example is given by the mitosis kernel $\wp=2\delta_{1/2}$ or the asymmetrical division $\wp=\delta_\nu+\delta_{1-\nu}$ with $0<\nu<1/2.$
\item Consider a kernel which is absolutely continuous with respect to the Lebesgue measure close to the origin and such that
\[\exists p_0,r_0>0,\quad\wp(z)\sim p_0z^{r_0-1}\quad \text{when}\ z\to0.\]
Then $\wp$ satisfies Assumption~\eqref{as:wpr_lim} with $\underline r=-r_0.$
For instance the power law kernels $\wp(z)=(\nu+2)z^\nu$ belong to this class provided that $\nu>-1.$
\end{enumerate}

\

\noindent{\bf Existence of Perron eigenelements}

There are many existence and uniqueness results for the Perron eigenvalue problem set on $\bR_+.$
We can mention~\cite{HW89,SVBW} in which the direct problem~\eqref{eq:direct_Perron} is solved in the case of equal or asymmetrical mitosis with constant coefficients $\tau$ and $B.$
The full problem~\eqref{eq:direct_Perron}-\eqref{eq:dual_Perron} is treated in~\cite{PR05,M06,DG10} for more and more general coefficients. 
We will use the following result, which is a particular case of~\cite[Theorem 1]{DG10}:

\begin{theorem}\label{th:recall}
Assume that {\bf(H$\tau$-H$B$-H$\wp$)} are satisfied.
There exist a unique solution (in the distributional sense) $(\lambda,G)\in\bR\times L^1(\bR_+)$ to the Perron eigenvalue problem~\eqref{eq:direct_Perron}
and a unique dual eigenfunction $\phi\in W^{1,\infty}_{loc}(\bR_+)$ such that $(\lambda,\phi)$ satisfies~\eqref{eq:dual_Perron} (in the sense of a.e. equality).
Moreover $\lambda>0.$
\end{theorem}

In section~\ref{sec:Perroneigenelements} we prove fine estimates on the profile $\phi(x),$ which are needed for the proof of our main theorem.
We do not need estimates on the profile $G(x),$
but the interested reader can find some in~\cite{DG10,BCG}.

\

\noindent{\bf The main result}

Define the {\it rescaled growth-fragmentation} operator
\[\cA g:=-(\tau g)'-\lambda g+\cF g\]
with domain
\[D(\cA)=\bigl\{g\in L^1_\phi\,|\, (\tau g)'\in L^1_\phi,\, (\tau g)(0)=0\bigr\}\]
and consider the abstract Cauchy problem
\begin{equation}\label{eq:abstract_resc}\left\{\begin{array}{l}
\dfrac{\d}{\d t}\,g=\cA g
\vspace{3mm}\\
g(0)=f^{\rm in}.
\end{array}\right.\end{equation}
We will prove in Section~\ref{sec:semigroups} that the operator $\cA$ generates a strongly continuous semigroup (also called $C_0$-semigroup) $(T_t)_{t\geq0}$ on $L^1_\phi.$
This result ensures that there exists a unique (mild) solution to the abstract Cauchy problem~\eqref{eq:abstract_resc} given by $g(t)=T_tf^{\rm in}.$
As a direct consequence for any $f^{\rm in}\in L^1_\phi,$ Equation~\eqref{eq:GF} admits a unique solution given by
\[f(t,\cdot)=\e^{\lambda t}g(t)=\e^{\lambda t}T_tf^{\rm in}.\]
Clearly the Perron eigenfunction $G$ is a steady-state for~\eqref{eq:abstract_resc} since by definition $\cA G=0.$
In other words $G$ is a fixed point for $(T_t)_{t\geq0},$ {\it i.e.} $T_tG=G$ for all time $t.$
The dual eigenfunction provides a conservation law for $(T_t)_{t\geq0}$
\begin{equation}\label{eq:cons_law}
\forall g\in L^1_\phi,\ \forall t\geq0,\qquad\langle T_tg,\phi\rangle=\langle g,\phi\rangle.
\end{equation}
This motivates the definition of the projection
\[Pg:=\langle g,\phi\rangle\,G.\]
In~\cite{MMP2}, Perthame {\it et al.} prove by using General Relative Entropy techniques that
\[\lim_{t\to+\infty}\|T_tg-Pg\|_{L^1_\phi}=0\qquad\text{for all}\ g\in L^1_\phi.\]
This is the asynchronous exponential growth property.
In the vocabulary of semigroups, it is called {\it strong convergence} of the semigroup $(T_t)_{t\geq0}$ to the projection $P.$
In terms of spectral theory, it ensures that in $L^1_\phi$ the Perron eigenvalue $\lambda$ is simple ({\it i.e.} has algebraic multiplicity one) and strictly dominant.

A stronger concept of convergence is the {\it uniform convergence}, {\it i.e.} the convergence for the norm of operators.
The uniform convergence of $(T_t)_{t\geq0}$ to $P$ would ensure the existence of a spectral gap.
Indeed it is a standard result (see~\cite[Prop. V.1.7]{EN} for instance) that for $C_{0}$-semigroups the uniform convergence is equivalent to the {\it uniform exponential convergence}.
Our main result is that the uniform convergence does not hold.

\begin{theorem}\label{th:main}
Under Hypotheses {\bf(H$\tau$-H$B$-H$\wp$)}, the semigroup $(T_t)_{t\geq0}$ does not converge uniformly to the projection $P.$
More precisely we have
\[\forall t\geq0,\qquad\| T_t-P\|_{\cL(L^1_\phi)}\geq1.\]
\end{theorem}

\

\noindent{\bf Structure of the paper}\\
We start by giving fine estimates on the dual eigenfunction $\phi$ in Section~\ref{sec:Perroneigenelements}.
These estimates are crucial all along the paper since we work in the space $L^1_\phi.$
In Section~\ref{sec:semigroups} we prove that the rescaled growth-fragmentation operator generates a positive contraction semigroup which provides the unique mild solution to our abstract Cauchy problem.
This semigroup is given by a Dyson-Phillips series and this is a central ingredient to prove the main theorem in the case when the fragmentation kernel has a support away from zero.
To make this more precise we define
\[z_0:=\inf\Supp\wp\in[0,1).\]
In Section~\ref{sec:z0>0} we prove the main theorem in the case $z_0>0$ by taking advantage of the Dyson-Phillips series which enables us to build solutions to our Cauchy problem that converge arbitrarily slowly to equilibrium, so that we can compute precisely the operator norm of $T_{t}-P$.
In Section~\ref{sec:z0=0} we extend the result to the case $z_0=0$ by using an accurate truncation of the fragmentation kernel and a passage to the limit which relies on the notion of quasi-compactness of a semigroup, with the crucial remark that it is invariant up to equivalent norm.
Finally in Section~\ref{sec:complementary} we give some results which complete the main theorem.

%
%

\section{The dual eigenfunction $\phi.$}\label{sec:Perroneigenelements}

To prove Theorem~\ref{th:main}, we need fine estimates on the dual eigenfunction $\phi$ given by the following theorem.

\begin{theorem}\label{th:phi_estimates}
Under Hypotheses {\bf(H$\tau$-H$B$-H$\wp$)}, there exists a constant $C>0$ such that
\[\forall x\geq0,\qquad \frac1C(1+x)^k\leq\phi(x)\leq C(1+x)^k,\]
where $k<1$ is uniquely defined by
\begin{equation}\label{eq:defk}
\wp_k=1+\frac{\lambda}{B_\infty}.
\end{equation}
\end{theorem}

The existence of a $k$ satisfying~\eqref{eq:defk} is guaranteed by condition~\eqref{as:wpr_lim}.
Notice that we recover the result in~\cite[Theorem 3.1]{PR05} in the case $\wp=2\delta_\frac12.$
The proof of Theorem~\ref{th:phi_estimates} is based on a truncated problem on $[0,L]$ and uses a maximum principle.
We need to consider two truncated problems:
\begin{equation}\label{eq:phi-}
-\tau(x)(\phi^-_L)'(x)+(B(x)+\lambda^-_L)\phi^-_L(x)=B(x)\int_0^1\phi^-_L(zx)\wp(\d z),\qquad \phi^-_L(L)=0,
\end{equation}
and
\begin{align}
-\tau(x)(\phi^+_L)'(x)+(B(x)+\lambda^+_L)\phi^+_L(x)=B(x)\int_0^1\phi^+_L(zx)\wp(\d z)+\frac1L\indic_{0\leq x\leq1}\phi^+_L(x),\nonumber\\
\phi_L^+(L)=0.\label{eq:phi+}
\end{align}
The existence and uniqueness of a solution for these truncated problems can be obtained by using the Krein-Rutman theorem (see~\cite{DG10} for more details).
When $L$ tends to $+\infty$ the solution to the truncated problem converges to the solution to~\eqref{eq:dual_Perron}.

\begin{theorem}\label{th:truncated}
There exist $K>0$ and $q>0,$ independent of $L,$ such that
\[\forall L>0,\ \forall x\in[0,L],\qquad \phi_L^\mp(x)\leq K(1+x^q).\]
Additionally when $L\to+\infty$ we have the convergences
\[\lambda^\mp_L\to\lambda,\]
\[\forall A>0,\qquad\phi^\mp_L\to\phi\quad\text{uniformly on}\ [0,A].\]
\end{theorem}

\begin{proof}
We only recall the main arguments and we refer to~\cite{DG10} for the details.
The uniform bound on $\phi_L^\mp$ is obtained by the same method we will use to prove Theorem~\ref{th:phi_estimates}.
It is based on the maximum principle stated in Lemma~\ref{lm:max_principle} below.
This uniform estimate combined with uniform bounds on $\lambda_L^\mp$ ensure compactness of the families ${(\lambda^\mp_L)}_{L>1}$ and ${(\phi^\mp_L)}_{L>1}.$
This provides the convergence of subsequences, and the uniqueness of the limit leads to the convergence of the entire family.
\end{proof}

The eigenvalue of first truncated problem approximates $\lambda$ from below, and the second from above.
This will be useful for the estimates on $\phi.$

\begin{lemma}
We have
\[\forall L>0,\quad\lambda^-_L<\lambda,\]
and there exists $L_0$ such that
\[\forall L\geq L_0,\quad\lambda^+_L>\lambda.\]
\end{lemma}

\begin{proof}
For $\lambda_L^-$ we have by integration of~\eqref{eq:phi-} against $G$
\[\lambda^-_L-\lambda=-\int_0^L\phi^-_L(y)\int_0^{\frac yL}B(\frac yz)G(\frac yz)\frac{\wp(\d z)}{z}\,\d y<0.\]
For $\lambda_L^+$ the integration of~\eqref{eq:phi+} against $G$ gives
\[\lambda^+_L-\lambda=\frac1L\int_0^1\phi^+_L(x)G(x)\,\d x-\int_0^L\phi^+_L(y)\int_0^{\frac yL}B(\frac yz)G(\frac yz)\frac{\wp(\d z)}{z}\,\d y.\]
To prove that the right hand side is positive for $L$ large enough, we use the convergence of $\phi_L^+$ which ensures that
\[\int_0^1\phi^+_L(x)G(x)\,\d x\xrightarrow[L\to+\infty]{}\int_0^1\phi(x)G(x)\,\d x>0\]
and for the second term we write, using Theorem~\ref{th:truncated},
\begin{align*}
L\int_0^L\phi^+_L(y)\int_0^{\frac yL}B(\frac yz)G(\frac yz)\frac{\wp(\d z)}{z}\,\d y&\leq KL\int_0^\infty(1+y^q)\int_0^{\frac yL}B(\frac yz)G(\frac yz)\frac{\wp(\d z)}{z}\,\d y\\
&=KL\int_0^1\int_{zL}^\infty(1+y^q)B(\frac yz)G(\frac yz)\frac{\d y}{z}\,\wp(\d z)\\
&=KL\int_0^1\int_L^\infty(1+z^qx^q)B(x)G(x)\,\d x\,\wp(\d z)\\
&\leq K\wp_0L\int_L^\infty(1+x^q)B(x)G(x)\,\d x\xrightarrow[L\to+\infty]{}0.
\end{align*}
The last term tends to zero because $G$ decreases faster than any powerlaw at infinity (see~\cite[Theorem 1.1]{DG10}).
\end{proof}

To get estimates on the truncated eigenfunctions $\phi^\pm$ we will use the following maximum principle. We refer the interested readers to~\cite[Lemma 3.2]{BCG} or~\cite[Appendix C]{DG10} for a proof of this result.

\begin{lemma}[Maximum principle]
\label{lm:max_principle}
Let $0<A<L$ and assume that $w\geq0$ on $[0,A],$ $w(L)\geq0$ and $w$ is a supersolution on $(A,L)$ in the sense that for all $x\in(A,L)$ we have
\[\mathcal S_L^\pm w(x):=-\tau(x)w'(x)+\lambda_L^{\pm}w(x)+B(x)w(x)-B(x)\int_0^1w(zx)\wp(\d z)>0.\]
Then $w\geq0$ on $[0,L].$
\end{lemma}

We are now ready to prove Theorem~\ref{th:phi_estimates}.
The behavior of $\phi$ at the origin is readily deduced from a result in~\cite{BCG}.
The more delicate point is the behavior at infinity and there is no result available in~\cite{BCG} in the case $B$ bounded.
The difficulty lies in the construction of super-solutions and sub-solutions and it requires the strong assumption~\eqref{as:B_infty}.
We will use that for any $r>\underline r$ and any $x>A_0$ (so that $B(x)=B_\infty$) we have, since $\lambda=B_\infty(\wp_k-1),$
\begin{align*}
\cS_L^\pm x^r&=-r\tau(x)x^{r-1}+\lambda_L^\pm x^r+B_\infty(1-\wp_r)x^r\\
&=-r\tau(x)x^{r-1}+[\lambda_L^\pm-\lambda+B_\infty(\wp_k-\wp_r)]x^r.
\end{align*}
We will combine this identity with $\lambda_L^-<\lambda<\lambda_L^+$ and the decay of the function $r\mapsto\wp_r.$

\begin{proof}[Proof of Theorem~\ref{th:phi_estimates}]
We split it into three steps.

\noindent{\it Step \#1: Convergence at the origin.}
Define
\[\Lambda(x):=\int_1^x\frac{\lambda+B(y)}{\tau(y)}\,\d y\]
the primitive of $\frac{\lambda+B}{\tau}$ which vanishes at $x=1$ (and thus is negative for $0<x<1$ as $\frac{\lambda+B}{\tau}$ is positive).
From~\cite[Theorem 1.10]{BCG} we know that $\phi(x)$ behaves like a positive constant times $\e^{-\Lambda(x)}.$
Condition~\eqref{as:tau0} ensures that $\Lambda(0)<+\infty,$ so $\phi(x)$ converges to a positive constant when $x\to0.$

\

\noindent{\it Step \#2: Upper bound at infinity.}
We start with the case $k\leq0$ and define $v(x)=x^k.$
For $x>A_0$ we have
\[\cS_L^+ v(x)=-k\tau(x)x^{k-1}+(\lambda_L^+-\lambda)x^k>0,\]
and for $x>1$ we have $\cS_L^+\phi_L^+(x)=0.$
We deduce that if $L>A:=\max(A_0,1)$ then for any constant $C>0$ the function $Cv-\phi_L^+$ is a supersolution for $\cS_L^+$ on $(A,L).$
On the interval $[0,A],$ the function $v$ is bounded from below by $A^k>0$ and since $\phi_L^+\to\phi$ uniformly on $[0,A]$ and $\phi$ is bounded on $[0,A]$
we can find $C>0$ such that $Cv\geq\phi_L^+$ on $[0,A]$ for all $L$ large enough.
For $x=L$ we have $Cv(L)=CL^k>\phi_L^+(L)=0.$
The assumptions of Lemma~\ref{lm:max_principle} are satisfied for $w=Cv-\phi_L^+,$ $A=\max(1,A_0)$ and $L$ large enough.
The maximum principle ensures that $\phi_L^+\leq Cv$ on $[0,L]$ and then $\phi\leq Cv$ on $\bR_+$ by passing to the limit $L\to+\infty.$

\medskip

When $k>0$ we need to modify a bit the function $v$ to get a supersolution.
Define $v(x)=x^k-x^{k-\epsilon}+1$ with $0<\epsilon<\min(k,1-\alpha),$ where $\alpha$ is defined in~\eqref{as:tau_infty}.
This function $v$ is bounded from below by a positive constant.
We compute
\begin{align*}
\cS_L^+ v(x)&=-\tau(x)v'(x)+(\lambda_L^+-\lambda)v(x)-B_\infty(\wp_k-\wp_{k-\epsilon})x^{k-\ep}+B_\infty(\wp_k-\wp_0)\\
&>-\tau(x)v'(x)+B_\infty(\wp_{k-\epsilon}-\wp_k)x^{k-\ep}+B_\infty(\wp_k-\wp_0).
\end{align*}
When $x\to+\infty,$ the dominant term in the last line above is $B_\infty(\wp_{k-\epsilon}-\wp_k)x^{k-\epsilon}.$
Indeed we have chosen $\epsilon$ such that $k-\epsilon>0$ and $k-\epsilon>k-1+\alpha,$ and when $x\to+\infty$ we have $\tau(x)v'(x)\sim\tau kx^{k-1+\alpha}.$
Since $r\mapsto\wp_r$ is a decreasing function, this dominant term is positive.
We deduce that we can find $A$ large enough such that for any $L>A$ and any $C>0,$ $Cv-\phi_L^+$ is a supersolution of $\cS^+_L$ on $(A,L).$
We conclude as in the case $k\leq0$ that there exists $C>0$ such that $\phi(x)\leq Cv(x)$ for all $x\geq0.$

\

\noindent{\it Step \#3: Lower bound at infinity.}
Choose $\epsilon\in(0,1-\alpha)$ such that $k-2\epsilon>\underline r$ and define $v_L(x)=(x^k+x^{k-\epsilon}-x^{k-2\epsilon})(1-\frac xL).$
We write
\[\cS_L^- v_L(x)=\cS_L^- v(x)+\frac1L\cS_L^- \tilde v(x)\]
where $v(x)=x^k+x^{k-\epsilon}-x^{k-2\epsilon}$ and $\tilde v(x)=-x^{k+1}-x^{k+1-\epsilon}+x^{k+1-2\epsilon},$
and we compute
\begin{align*}
\cS_L^- v(x)&=-\tau(x)v'(x)+(\lambda_L^--\lambda)(x^k+x^{k-\epsilon})+B_\infty(\wp_k-\wp_{k-\epsilon})x^{k-\epsilon}\\
&\hspace{60mm} -[\lambda_L^-+B_\infty(1-\wp_{k-2\epsilon})]x^{k-2\epsilon}\\
&<-\tau(x)v'(x)+[B_\infty(\wp_k-\wp_{k-\epsilon}) +x^{-\epsilon}B_\infty\wp_{k-2\epsilon}]x^{k-\epsilon}.
\end{align*}
The last line is equivalent to $B_\infty(\wp_k-\wp_{k-\epsilon})x^{k-\epsilon}$ when $x\to+\infty.$
We deduce that $\cS_L^-v(x)<0$ for $x$ large enough since $\wp_k<\wp_{k-\epsilon}.$
Using the convergence $\lambda_L^-\to\lambda$ when $L\to+\infty$ and the decay of $r\mapsto\wp_r$ we have
\begin{align*}
\cS_L^-\tilde v(x)&=-\tau(x)\tilde v'(x)+(B_\infty(\wp_{k+1}-\wp_{k})+\lambda-\lambda_L^-)x^{k+1}\\
&\hspace{2mm} +(B_\infty(\wp_{k+1-\epsilon}-1)-\lambda_L^-)x^{k+1-\epsilon}-(B_\infty(\wp_{k+1-2\epsilon}-1)-\lambda_L^-)x^{k+1-2\epsilon}\\
&\hspace{-4mm}\underset{x\to+\infty}{\sim}\underbrace{(B_\infty(\wp_{k+1}-\wp_{k})+\lambda-\lambda_L^-)}_{<0\ \text{for $L$ large}}x^{k+1}.
\end{align*}
We deduce that for any $c>0$ and for $A$ large enough, the function $\phi_L^--c\,v_L$ is a supersolution to $\cS_L^-$ on $(A,L)$ for any $L>A.$
Since $v_L$ is upper bounded on $[0,A]$ and $v_L(L)=\phi_L^-(L)=L$ we can conclude
by arguing as in Step~2 that there exists $c>0$ small enough such that $\phi_L^-\geq c\,v_L$ on $[0,L]$ for $L$ large enough.
Passing to the limit $L\to+\infty$ we get $\phi(x)\geq c\,v(x)$ for all $x\geq0$ and then $\phi(x)\geq c\,x^k$ for all $x\geq1.$

Notice that when $k\geq0$ we can consider the simpler subsolution $v(x)=x^k(1-\frac xL).$

\end{proof}

%
%

\section{Semigroup and mild solution}\label{sec:semigroups}

In this section we prove that the operator $(\cA,D(\cA))$ generates a positive $C_0$-semigroup.
This ensures that the abstract Cauchy problem~\eqref{eq:abstract_resc} admits a unique mild solution.
Recall that the rescaled growth-fragmentation operator $\cA$ is defined by
\[\cA g=-(\tau g)'-\lambda g-Bg+\cF_+ g\]
with domain
\[D(\cA)=\bigl\{g\in L^1_\phi\,|\, (\tau g)'\in L^1_\phi,\, (\tau g)(0)=0\bigr\}.\]
First we need to check that $\cA g$ is well defined for $g\in D(\cA).$
Clearly $(\tau g)',$ $\lambda g$ and $Bg$ are well defined and belong to $L^1_\phi$ when $g\in D(\cA).$
The only term which requires attention is $\cF_+g.$
Since we consider kernels $\wp$ which are not necessarily absolutely continuous with respect to the Lebesgue measure,
the definition~\eqref{def:cF+} of $\cF_+$ does not have a classical sense for an arbitrary function $f$ in $L^1_\phi,$ or even in $D(\cA).$
It is well defined for functions in the dense subspace
\[\cX:=\bigl\{g\in L^1_\phi,\ Bg\in C_c(\bR_+)\bigr\}.\]
Lemma~\ref{lm:F+def} below ensures that the operator $\mathcal F_+:\cX\subset L^1_\phi\to L^1_\phi$ defined by~\eqref{def:cF+} can be uniquely extended into a continuous operator on $L^1_\phi.$
From now on when talking about the operator $\mathcal F_+$ we mean this extension and then $\cA g$ is well defined for $g\in D(\cA).$

\begin{lemma}\label{lm:F+def}
There exists a unique bounded operator $\mathcal F_+:L^1_\phi\to L^1_\phi$ such that~\eqref{def:cF+} holds for any $f\in \cX.$
Additionally
\[\|\mathcal F_+\|_{\cL(L^1_\phi)}\leq C^2\max(\wp_0,\wp_k)\|B\|_\infty\]
where $C$ is the constant which appears in Theorem~\ref{th:phi_estimates}.
\end{lemma}

\begin{proof}
It is a consequence of the continuous extension theorem, provided we prove
\[\forall f\in \cX,\qquad\biggl\|\int_0^1B\Bigl(\frac{\cdot}{z}\Bigr)f\Bigl(\frac{\cdot}{z}\Bigr)\frac{\wp(\d z)}{z}\biggr\|_{L^1_\phi}\leq C^2\max(\wp_0,\wp_k)\|B\|_\infty\|f\|_{L^1_\phi}.\]
Choose $f\in \cX.$
If $k\geq0,$ Theorem \ref{th:phi_estimates} implies
\begin{align*}
\biggl\|\int_0^1 B\Bigl(\frac{\cdot}{z}\Bigr)f\Bigl(\frac{\cdot}{z}\Bigr)\frac{\wp(\d z)}{z}\biggr\|_{L^1_\phi}&\leq\int_0^\infty\int_0^1 B\Bigl(\frac xz\Bigr)\Bigl|f\Bigl(\frac xz\Bigr)\Bigr|\frac{\wp(\d z)}{z}\phi(x)\d x\\
&\leq\int_0^1\int_0^\infty B(y)|f(y)|\phi(zy)\,\d y\,\wp(\d z)\\
&\leq \|B\|_\infty C\int_0^\infty|f(y)|\int_0^1(1+zy)^k\wp(\d z)\,\d y\\
&\leq \|B\|_\infty C\int_0^\infty|f(y)|(1+y)^k\biggl(\int_0^1\wp(\d z)\biggr)\,\d y\\
&\leq \|B\|_\infty C^2\wp_0\|f\|_{L^1_\phi},
\end{align*}
and if $k<0$
\begin{align*}
\biggl\|\int_0^1 B\Bigl(\frac{\cdot}{z}\Bigr)f\Bigl(\frac{\cdot}{z}\Bigr)\frac{\wp(\d z)}{z}\biggr\|_{L^1_\phi}
&\leq \|B\|_\infty C\int_0^\infty|f(y)|\int_0^1(1+zy)^k\wp(\d z)\,\d y\\
&\leq \|B\|_\infty C\int_0^\infty|f(y)|(1+y)^k\biggl(\int_0^1z^k\wp(\d z)\biggr)\,\d y\\
&\leq \|B\|_\infty C^2\wp_k\|f\|_{L^1_\phi}.
\end{align*}
The conclusion follows from the fact that $\max(\wp_0,\wp_k)=\wp_0$ if and only if $k\geq0.$
\end{proof}

Lemma~\ref{lm:F+def} ensures that $\cA$ is a bounded perturbation of the transport operator
\[\cA_0g:=-(\tau g)'-\lambda g-Bg\]
with domain
\[D(\cA_0)=D(\cA).\]
To prove that $\cA=\cA_0+\cF_+$ generates a $C_0$-semigroup we first prove that $\cA_0$ generates a $C_0$-semigroup
and we then use a bounded perturbation theorem.

\begin{proposition}\label{prop:gen_St}
The transport operator $(\cA_0,D(\cA_0))$ generates a positive contraction $C_0$-semigroup $(S_t)_{t\geq0}.$
\end{proposition}

\begin{proof}
We use the Lumer-Philipps theorem (see for instance~\cite[Theorem II.3.15]{EN}).
Clearly $D(\cA_0)$ is dense in $L^1_\phi.$
It remains to prove that for all $\mu>0,$
\[\begin{array}{c}
\forall g\in D(\cA_0),\qquad\|(\mu-\cA_0)g\|_{L^1_\phi}\geq\mu\|g\|_{L^1_\phi}\qquad\text{(dissipativity)}
\vspace{1mm}\\
\text{and}\quad\forall h\in L^1_\phi,\,\exists g\in D(\cA_0),\qquad (\mu-\cA_0)g=h\qquad\text{(surjectivity)}
\end{array}\]
or equivalently
\[\forall h\in L^1_\phi,\,\exists g\in D(\cA_0),\qquad (\mu-\cA_0)g=h\quad\text{and}\quad\|h\|_{L^1_\phi}\geq\mu\|g\|_{L^1_\phi}.\]
Let $\mu>0$ and $h\in L^1_\phi.$
We want to solve
\be\label{eq:ghmu}\mu g+(\tau g)'+\lambda g+Bg=h,\qquad x>0,\ee
with the condition $(\tau g)(0)=0.$
We obtain
\be\label{eq:ghmu2}\tau(x)g(x)=\int_0^x \e^{-\int_y^x\frac{\mu+\lambda+B(z)}{\tau(z)}dz}h(y)\,\d y.\ee
We need to verify that $g$ thus defined belongs to $D(\cA_0).$
Recall that we have defined in the proof of Theorem~\ref{th:phi_estimates} the function
\[\Lambda(x)=\int_1^x\frac{\lambda+B(y)}{\tau(y)}\,\d y.\]
A direct computation gives that the function $x\mapsto \phi(x)\e^{-\Lambda(x)}$ is decreasing since its derivative is equal to $-\frac1\tau\cF_+^*\phi.$
Using this property we get
\begin{align*}
\int_0^\infty |g(x)|\phi(x)\,\d x&\leq\int_0^\infty\frac{\phi(x)}{\tau(x)}\e^{-\Lambda(x)} \int_0^x|h(y)|\e^{\Lambda(y)}\e^{-\int_y^x\frac{\mu}{\tau(z)}\,dz}\d y\d x\\
&=\int_0^\infty |h(y)|\e^{\Lambda(y)} \int_y^\infty \frac{\phi(x)}{\tau(x)}\e^{-\Lambda(x)}\e^{-\int_y^x\frac{\mu}{\tau(z)}\,dz}\d x\d y\\
&\leq\int_0^\infty |h(y)|\phi(y) \int_y^\infty \frac{1}{\tau(x)}\e^{-\int_y^x\frac{\mu}{\tau(z)}\,dz}\d x\d y\\
&=\frac1\mu\|h\|_{L^1_\phi}
\end{align*}
so $g\in L^1_\phi$ and $\|g\|_{L^1_\phi}\leq\frac1\mu\|h\|_{L^1_\phi}.$
It remains to check that $(\tau g)'\in L^1_\phi.$
Using Equation~\eqref{eq:ghmu} it is an immediate consequence of the fact that $g,h\in L^1_\phi$ and that $B$ is bounded.

The positivity of the semigroup is a consequence of the positivity of the resolvent $(\mu-\cA_0)^{-1}$ (see for instance~\cite[Characterization Theorem VI.1.8]{EN}), which is obvious from~\eqref{eq:ghmu2}.
\end{proof}

Since $\cF_+$ is a positive bounded perturbation of $\cA_0,$ we have the following corollary.

\begin{corollary}
The rescaled growth-fragmentation operator $(\cA,D(\cA))$ generates a positive contraction $C_0$-semigroup $(T_t)_{t\geq0}.$
This semigroup can be obtained from $(S_t)_{t\geq0}$ by the Dyson-Phillips series
\[T_t=\sum_{n=0}^\infty S_t^{(n)}\]
where $S_t^{(0)}:=S_t$ and
\[S_t^{(n)}:=\int_0^t S_{t-s}\cF_+S_s^{(n-1)}\,\d s.\]
\end{corollary}

\begin{proof}
The bounded perturbation theorem III-1.3 in~\cite{EN} ensures that $\cA$ is the generator of a $C_0$-semigroup.
The positivity of this semigroup follows from the positivity of $(S_t)_{t\geq0}$ and $\cF_+.$
For the contraction we use~\cite[Proposition {\bf1.}8.22]{Kavian} and the definition of $\phi$ to write for $g\in D(\cA)$
\begin{align*}
\langle\cA g,(\Sign g)\phi\rangle=\int_0^\infty\Sign g\,(\cA_0g+\cF_+g)\phi&=\int_0^\infty(\cA_0|g|+\Sign g\,\cF_+g)\phi\\
&\leq\int_0^\infty(\cA_0|g|+\cF_+|g|)\phi=\int_0^\infty |g|\cA^*\phi=0.
\end{align*}
The formulation in terms of the Dyson-Phillips series is ensured by~\cite[Theorem III-1.10]{EN}.
\end{proof}

The last result of this section, given by~\cite[Proposition II-6.4]{EN}, concerns the existence of a unique solution to the abstract Cauchy problem~\eqref{eq:abstract_resc}.

\begin{corollary}
For every $f^{\rm in}\in L^1_\phi$ the orbit map
\[g:t\mapsto g(t)=T_tf^{\rm{in}}\]
is the unique \emph{mild} solution to the abstract Cauchy problem~\eqref{eq:abstract_resc}.
In other words it is the unique continuous function $\bR_+\to L^1_\phi$ which satisfies $\int_0^tg(s)\d s\in D(\cA)$ for all $t\geq0$ and
\[g(t)=\cA\int_0^tg(s)\,\d s+f^{\rm in}.\]
\end{corollary}

\

%
%

\section{Non-uniform convergence in the case $z_0>0$}\label{sec:z0>0}

The aim of this section is to establish the following theorem, which provides a more precise result than Theorem~\ref{th:main} in the case $z_0=\inf\Supp\wp>0.$
\begin{theorem}\label{th:norm=2}
We assume {\bf(H$\tau$-H$B$-H$\wp$)} and that $\wp$ is such that $z_0>0.$
Then
\[\forall t\geq0,\qquad\|T_t-P\|_{\mathcal L(L^1_\phi)}=2.\]
\end{theorem}

\

To prove this theorem, we start with two lemmas.
Define $g_a(x)=\frac{1}{\phi(x)}\indic_{a\leq x\leq a+1}.$

\begin{lemma}\label{lm:supp_St}
Assume that $\wp$ is such that $z_0>0.$
Then for all $X>0$ we have
\[\forall n<\log_{1/z_0}\frac aX,\qquad\int_0^XS_t^{(n)}g_a(x)\phi(x)\,\d x=0.\]
\end{lemma}

\begin{proof}
We prove by induction on $n$ that $\Supp(S_t^{(n)}g_a)\subset[z_0^na,+\infty).$

For $n=0$ the results follows from the fact that $S_t$ is the semigroup of a transport equation with positive speed.

To go from $n$ to $n+1$ we check that if $g$ is a function such that $\Supp g\subset[A,+\infty)$ then $\Supp\cF_+ g\subset[z_0 A,+\infty).$
Indeed we have
\[\cF_+ g(x)=\int_{z_0}^1B\Bigl(\frac xz\Bigr)g\Bigl(\frac xz\Bigr)\frac{\wp(\d z)}{z}\]
and $\cF_+ g(x)=0$ when $\frac x{z_0}<A.$
\end{proof}

\begin{lemma}\label{lm:Dyson-Philipps}
For all $n\in\bN$ and all $t\geq0$ we have
\[\|S_t^{(n)}\|_{\cL(L^1_\phi)}\leq\frac{\|\cF_+\|^n\, t^n}{n!}.\]
\end{lemma}

\begin{proof}
It follows from the definition of $S^{(n)}_t$ by an induction on $n,$ using that $\cF_+$ is bounded and that $\|S_t\|_{\cL(L^1_\phi)}\leq1$ since $(S_t)_{t\geq0}$ is a contraction.
\end{proof}

We are now ready to prove Theorem~\ref{th:norm=2}.

\begin{proof}[Proof of Theorem~\ref{th:norm=2}]
First we have $\|T_t-P\|_{\cL(L^1_\phi)}\leq\| T_t\|_{\cL(L^1_\phi)}+\| P\|_{\cL(L^1_\phi)}\leq2$ because $(T_t)$ is a contraction and $P$ is a projection.

For the other inequality we consider the initial distribution $f^{\rm{in}}=g_a$ which satisfies $\|g_a\|_{L^1_\phi}=\langle g_a,\phi\rangle=1$ and we write for $X>0$ to be chosen later
\begin{align*}
\|T_t g_a-P g_a\|_{L^1_\phi}&=\int_0^X|T_tg_a(x)-G(x)|\phi(x)\,\d x
+\int_X^\infty|T_tg_a(x)-G(x)|\phi(x)\,\d x\\
&\geq\int_0^XG(x)\phi(x)\,\d x-\int_0^XT_tg_a(x)\phi(x)\,\d x\\
&\hspace{20mm}+\int_X^\infty T_tg_a(x)\phi(x)\,\d x-\int_X^\infty G(x)\phi(x)\,\d x.
\end{align*}
Let $\epsilon>0$ and $X>0$ be such that $\int_0^XG(x)\phi(x)\d x\geq1-\epsilon,$ and so $\int_X^\infty G(x)\phi(x)\d x\leq\epsilon.$
Consider $g_a(x)=\frac{1}{\phi(x)}\indic_{a\leq x\leq a+1}.$
Using Lemma~\ref{lm:supp_St} and Lemma~\ref{lm:Dyson-Philipps} we have
\[\int_0^XT_tg_a(x)\phi(x)\,\d x\leq R_t(a)\]
where
\[R_t(a):=\sum_{n\geq\log_{1/z_0}(\frac aX)}\frac{\|\cF_+\|^n t^n}{n!}\]
Since the series is absolutely convergent and $a\mapsto\log_{1/z_0}\frac aX$ converges to $+\infty$ when $a\to+\infty,$
we deduce that $\lim_{a\to+\infty}R_t(a)=0$ and we can find $a>X$ large enough so that $\int_0^XT_tg_a(x)\phi(x)\,dx\leq\epsilon.$
Since the conservation law~\eqref{eq:cons_law} ensures that $\langle g(t),\phi\rangle=\langle g_a,\phi\rangle=1,$ we have $\int_X^{+\infty}T_tg_a(x)\phi(x)\,dx\geq1-\epsilon$ and finally we get
\[\|T_tg_a-P g_a\|_{L^1_\phi}\geq 2-4\epsilon.\]
The proof is complete since $\epsilon>0$ is arbitrary and $\|g_a\|_{L^1_\phi}=1.$
\end{proof}

\

%
%
\section{The general case}\label{sec:z0=0}

In the previous section we have established Theorem~\ref{th:main} in the case of a fragmentation kernel which satisfies $z_0=\inf\Supp\wp>0$
by computing exactly the norm of the operator $T_t-P.$
This strategy cannot be applied when $z_0=0.$
To treat this case we use the notion of quasi-compactness.

\begin{definition}
Let $\left(U_t\right)_{t\geq0}$ be a $C_{0}$-semigroup on a Banach space $E.$
It is said to be quasi-compact if and only if there exists a compact operator $K$ and $t_{0}\geq0$ such that
$$
\left\| U_{t_{0}} -K\right\|_{\cL(E)}<1.
$$
\end{definition}

The projection $P$ being a compact operator (since it is rank one),
if we prove that $(T_t)$ is not quasi-compact then the result of Theorem~\ref{th:main} follows immediately.
To do so we start by truncating the fragmentation kernel in such a way that it satisfies $z_0>0.$
For this truncated problem we use Theorem~\ref{th:norm=2} to check that the semigroup is not quasi-compact.
Then we prove that this lack of quasi-compactness is preserved as the truncation parameter vanishes.
The main difficulty in this procedure is that the dual eigenfunction $\phi$ is modified when we truncate the fragmentation kernel.
Since this function is the weight of the $L^1$ space in which we work, we will need the following proposition.

\begin{proposition}\label{prop:invariance}
The notion of quasi-compactness is invariant under change of equivalent norm.
\end{proposition}

\begin{proof}
This is a direct consequence of~\cite[Proposition V-3.5]{EN}, which states that a semigroup $(U_t)$ is quasi-compact if and only if 
$$
\lim_{t\rightarrow +\infty}\inf\,\bigl\{ \left\| U_{t} - K\right\|_{\cL\left(E\right)} |\ K\ \mbox{compact} \bigr\}=0.
$$
\end{proof}

\subsection{The lack of quasi-compactness in the case $z_0>0$}

We use Theorem~C-IV-2.1, p.343  in~\cite{Nagel86} which is recalled here.

\begin{theorem}\label{th:qc_convergence}
Let $(U_t)_{t\geq0}$ be a positive $C_0$-semigroup on a Banach lattice $E$
which is bounded, quasi-compact and has spectral bound zero.
Then there exists a positive projection $Q$ of finite rank and suitable constants $M\geq1$ and $a>0$ such that
\[\|U_t-Q\|_{\mathcal L(E)}\leq M\e^{-at}\quad\text{for all}\ t\geq0.\]
\end{theorem}

\begin{corollary}\label{cor:nonqc_z0>0}
If $\wp$ is such that $z_0>0,$ then the rescaled growth-fragmentation semigroup $(T_t)_{t\geq0}$ is not quasi-compact.
\end{corollary}

\begin{proof}
We refer to~\cite{Nagel86} for the definitions of a Banach lattice and the spectral bound of a semigroup.
The following properties are readily deduced from the definitions:

- $L^1_\phi$ is a Banach lattice;

- a contraction semigroup is bounded and has nonpositive spectral bound;

- a semigroup which admits a nonzero fixed point has a nonnegative spectral bound.

The semigroup $(T_t)_{t\geq0}$ is a contraction semigroup with a positive fixed point $G.$
So the only missing condition on $(T_t)_{t\geq0}$ to apply Theorem~\ref{th:qc_convergence} is the quasi-compactness.
We use this to prove by contradiction that $(T_t)_{t\geq0}$ cannot be quasi-compact.

Assume that $(T_t)_{t\geq0}$ is quasi-compact.
By Theorem~\ref{th:qc_convergence} we deduce that it converges uniformly to a positive projection $Q.$
If we can prove that $Q=P$ then we get a contradiction with Theorem~\ref{th:norm=2} and the result follows.
The end of the proof consists in proving this identity.

Let $f\in L^1_\phi$ and consider its positive part $f^+.$
Since the projection $Q$ is positive we have that $Qf^+$ is positive.
The conservation law~\eqref{eq:cons_law} which states that $\langle T_tf^+,\phi\rangle=\langle f^+,\phi\rangle$
implies by passing to the limit $t\to+\infty$ that $\langle Qf^+,\phi\rangle=\langle f^+,\phi\rangle,$ where $\langle\cdot,\cdot\rangle$ is defined in 
$(\ref{def:crochets})$.
From the semigroup property $T_{t+s}=T_tT_s$ we get by passing to the limit $s\to+\infty$ that $Q=T_tQ.$
Applying this identity to $f^+$ we get $T_tQf^+=Qf^+.$
We have proved that $Qf^+$ is a positive fixed point of $(T_t)$ which satisfies $\langle Qf^+,\phi\rangle=\langle f^+,\phi\rangle.$
These properties together with the uniqueness of the Perron eigenfunction $G$ in Theorem~\ref{th:recall} imply that  $Qf^+=\langle f^+,\phi\rangle G=Pf^+.$
Similarly we have $Qf^-=Pf^-$ and then $Qf=Pf.$
This means that $Q=P$ and the proof is complete.
\end{proof}

\subsection{An accurate truncation}

The strategy for a kernel such that $z_0=0$ consists in truncating it by defining, for $\ep\leq\ep_0$ where $\ep_0>0$ is small enough so that $\int_{\ep_0}^1z\wp(dz)>0,$
\[\wp^\ep:=\frac1{\int_\ep^1z\wp(\d z)}\indic_{[\ep,1]}\wp.\]
This new kernel satisfies $z_0^\ep=\inf\Supp\wp^\ep\geq\ep>0,$ $\wp_1^\ep=\int_0^1z\wp^\ep(dz)=1$ and converges to $\wp$ when $\ep\to0.$
We denote by $\lambda_\ep$ and $\phi_\ep$ the Perron eigenvalue and the adjoint eigenfunction corresponding to $\wp^\ep.$
Corollary~\ref{cor:nonqc_z0>0} applies and ensures that the semigroup associated to $\wp^\ep$ is not quasi-compact in $L^1_{\phi_\ep}.$
But using Theorem~\ref{th:phi_estimates} we see that there is no reason for $\phi_\ep$ to be comparable $\phi$ at infinity,
so the corresponding weighted $L^1$-norms are not equivalent and $L^1_\phi\neq L^1_{\phi_\ep}.$
To overcome this problem we modify $B(x)$ by defining
\[B_{\eta,A}:=B+\eta\indic_{[A,+\infty)}\]
for $A\geq A_0$ and $\eta>-B_\infty,$ so that $B_{\eta,A}(x)>0$ for $x>0.$
We denote by $\lambda_{\ep,\eta,A}$ the Perron eigenvalue associated to the fragmentation coefficients $\wp^\ep$ and $B_{\eta,A}.$
The idea is to choose accurately $A$ and $\eta$ as functions of $\ep$ in such a way that the associated dual eigenfunction, denoted afterwards by $\hat\phi_\ep,$ is comparable to $\phi.$
We start with two useful lemma.

\begin{lemma}\label{lm:continuity_lambda}
The function $(\ep,\eta,A)\mapsto\lambda_{\ep,\eta,A}$ is continuous on $[0,\ep_0]\times(-B_\infty,+\infty)\times[A_0,+\infty).$
Additionally we have for any $(\ep,\eta)\in[0,1)\times(-B_\infty,+\infty),$
\be\label{eq:lim_lambda_A_infty}\lim_{A\to+\infty}\lambda_{\ep,\eta,A}=\lambda_\ep.\ee
\end{lemma}

\begin{proof}
This result follows from uniform estimates on $\lambda_{\ep,\eta,A}$ and the associated eigenfunctions $G_{\ep,\eta,A}$ and $\phi_{\ep,\eta,A}$ as obtained in~\cite{DG10}.
It allows to prove compactness of the families which, combined to the uniqueness of the eigenelements, provides the continuity with respect to parameters and the limit~\eqref{eq:lim_lambda_A_infty}.
\end{proof}

\begin{lemma}\label{lm:bound_lambda}
For any $(\ep,\eta,A)\in[0,\ep_0]\times(-B_\infty,+\infty)\times[A_0,+\infty)$ we have
\[0<\lambda_{\ep,\eta,A}\leq \|B\|_\infty +\eta.\]
\end{lemma}

\begin{proof}
Integrating the equation satisfied by $(\lambda_{\ep,\eta,A},G_{\ep,\eta,A})$ on $\bR^+$ we get
\[0<\lambda_{\ep,\eta,A}=\int_0^\infty(B(x)+\eta\indic_{x\geq A})G_{\ep,\eta,A}(x)\,\d x\leq \|B\|_\infty +\eta.\]
\end{proof}

Now we are ready to prove the following proposition, which leads to the definitions of $\hat\lambda_\ep,$ $\hat\phi_\ep$ and $(\hat T_t^\ep)_{t\geq0}.$

\begin{proposition}\label{prop:eta_Aeta}
For any $k>\underline r$ and any $\ep\in(0,\ep_0)$ there exist $\eta>-B_\infty$ and $A_\eta>A_0$ such that
\[\wp_k^\ep=1+\frac{\lambda_{\ep,\eta,A_\eta}}{B_\infty+\eta},\]
and $|\eta|\to0$ when $\ep\to0.$
\end{proposition}

\begin{definition}
For any $\ep\in(0,\ep_0)$ we denote by $\hat\lambda_\ep,$ $\hat\phi_\ep$ and $(\hat T_t^\ep)_{t\geq0}$ respectively the Perron eigenvalue, the dual eigenfunction and the semigroup
corresponding to $\wp^\ep$ and $B_{\eta,A}$ with the choice of $\eta$ and $A=A_\eta$ in Proposition~\ref{prop:eta_Aeta}.
\end{definition}

\begin{proof}[Proof of Proposition~\ref{prop:eta_Aeta}]
Fix $\ep\in(0,\ep_0)$ and define
\[\wp_k^\ep=\int_0^1z^k\wp^\ep(\d z)=\frac{\int_\ep^1z^k\wp(\d z)}{\int_\ep^1z\wp(\d z)}>1.\]
We want to find $\eta>-B_\infty$ (small when $\ep$ is small) and $A>A_0$ (large) such that
\be\label{eq:lb_eta}
\lambda_{\ep,\eta,A}=(\wp_k^\ep-1)(B_\infty+\eta).
\ee
The right hand side does not depend on $A$ and is an affine function of $\eta.$
It tends to $0$ when $\eta\to-B_\infty$ and to $+\infty$ when $\eta\to+\infty.$ 
The strategy consists in proving that the left hand side is a continuous bounded function of $\eta$ provided we choose $A$ as a suitable function of $\eta.$
For any $\eta>-B_\infty$ we define the auxiliary number
\[\tilde\lambda_{\ep,\eta}:=\lambda_\ep+\frac{\ep}{1+|\eta|}(\lambda_{\ep,\eta,A_0}-\lambda_{\ep}).\]
It is easy to see that
\[\min(\lambda_\ep,\lambda_{\ep,\eta,A_0})<\tilde\lambda_{\ep,\eta}<\max(\lambda_\ep,\lambda_{\ep,\eta,A_0}).\]
From Lemma~\ref{lm:continuity_lambda} we know that the mapping $A\in(A_0,+\infty)\mapsto\lambda_{\ep,\eta,A}$ varies continuously from $\lambda_{\ep,\eta,A_0}$ to $\lambda_\ep.$
We deduce that there exists $A_\eta>A_0$ such that
\[\lambda_{\ep,\eta,A_\eta}=\tilde\lambda_{\ep,\eta}.\]
We get~\eqref{eq:lb_eta} with $A=A_\eta$ if we find $\eta$ such that
\[\tilde\lambda_{\ep,\eta}=(\wp_k^\ep-1)(B_\infty+\eta).\]
From Lemma~\ref{lm:continuity_lambda} we know that $\eta\mapsto\tilde\lambda_{\ep,\eta}$ is continuous.
We have the lower bound
\[\tilde\lambda_{\ep,\eta}\geq(1-\ep)\lambda_\ep>0\]
and using Lemma~\ref{lm:bound_lambda} we get the upper bound
\[\tilde\lambda_{\ep,\eta}\leq\lambda_\ep+\frac{\ep}{1+|\eta|}\lambda_{\ep,\eta,A_0}\leq\lambda_\ep+\ep\frac{\|B\|_\infty +\eta}{1+|\eta|}\leq\lambda_\ep+\ep\max(1,\|B\|_\infty ).\]
So necessarily the graph of $\eta\mapsto\tilde\lambda_{\ep,\eta}$ (when $\eta$ varies from $-B_\infty$ to $+\infty$) crosses the affine graph of $\eta\mapsto(\wp_k^\ep-1)(B_\infty+\eta).$
We choose $\eta$ which realizes such an intersection.

\medskip

The last step for completing the proof is to check that the $\eta$ that we have defined is small when $\ep$ is small.
The result of Lemma~\ref{lm:bound_lambda} for $\eta=0$ gives $\lambda_\ep\leq \|B\|_\infty .$
Using this estimate and the bounds on $\tilde\lambda_{\ep,\eta}$ above we get
\[-\ep \|B\|_\infty \leq\tilde\lambda_{\ep,\eta}-\lambda_\ep\leq\ep\max(\|B\|_\infty ,1).\]
This ensures that $\tilde\lambda_{\ep,\eta}\to\lambda$ when $\ep\to0$ and, since $\wp_k^\ep\to\wp_k,$
\[\eta=\frac{\wp_k^\ep-1}{\tilde\lambda_{\ep,\eta}}-B_\infty=\frac{\wp_k^\ep-1}{\tilde\lambda_{\ep,\eta}}-\frac{\wp_k-1}{\lambda}\to0.\]

\end{proof}

Proposition~\ref{prop:eta_Aeta} leads to the following corollaries.

\begin{corollary}\label{co:uniform_bound_phi_ep}
For any $\ep\in(0,\ep_0),$ there exist positive constants $c_\ep$ and $C_\ep$ such that
\[c_\ep\hat\phi_\ep\leq\phi\leq C_\ep\hat\phi_\ep.\]
As a consequence, $L^1_\phi=L^1_{\hat\phi_\ep}$ and the natural norms of these two spaces are equivalent.
\end{corollary}

\begin{proof}
This is a direct consequence of Theorem~\ref{th:phi_estimates} and Proposition~\ref{prop:eta_Aeta}.
\end{proof}

\begin{corollary}\label{co:nonqc_truncated}
For all $\ep\in(0,\ep_0),$ the semigroup $(\hat T_t^\ep)_{t\geq0}$ is not quasi-compact in $L^1_\phi.$
As a consequence
\[\forall t\geq0,\qquad\|\hat T_t^\ep-P\|_{\cL(L^1_\phi)}\geq1.\]
\end{corollary}

\begin{proof}
Corollary~\ref{co:uniform_bound_phi_ep} and Proposition~\ref{prop:invariance} ensure that if a semigroup is not quasi-compact in $L^1_{\hat\phi_\ep}$ then it is not quasi-compact in $L^1_\phi$ either.
But we know from Corollary~\ref{cor:nonqc_z0>0} that $(\hat T_t^\ep)_{t\geq0}$ is not quasi-compact in $L^1_{\hat\phi_\ep}$ since $z_0^\ep>0.$
\end{proof}

\subsection{The general case}

To get Theorem~\ref{th:main} in the case $z_0=0$ it suffices to prove the convergence in the following proposition and to use Corollary~\ref{co:nonqc_truncated}.

\begin{proposition}\label{prop:conv_Tt}
For any time $t\geq0$ we have the convergence
\[\| \hat T_t^\ep-T_t\|_{\cL(L^1_\phi)}\to0\qquad\text{when}\ \ep\to0.\]
\end{proposition}

The generator of $(\hat T_t^\ep)_{t\geq0}$ is given by
\[\hat\cA_\ep g=-(\tau g)'-\hat\lambda_\ep g+\hat\cF_\ep g\]
where $\hat\cF_\ep$ is defined on $L^1_\phi$ by
\[\hat\cF_\ep\,g(x)=\int_0^1B_{\eta,A_\eta}\Bigl(\frac xz\Bigr)g\Bigl(\frac xz\Bigr)\frac{\wp^\ep(\d z)}{z}-B_{\eta,A_\eta}(x)g(x).\]
Before proving Proposition~\ref{prop:conv_Tt} we give two lemmas which ensure the convergence of $\hat\cA_\ep$ to $\cA$ in a sense we will precise later.

\begin{lemma}\label{lm:conv_lb}
When $\ep\to0$ we have $\hat\lambda_\ep\to\lambda.$
\end{lemma}

\begin{proof}
At the end of the proof of Proposition~\ref{prop:eta_Aeta} we have proved that $\tilde\lambda_{\ep,\eta}\to\lambda.$
But by definition of $\eta$ and $A_\eta,$ $\tilde\lambda_{\ep,\eta}=\hat\lambda_\ep.$
\end{proof}

\begin{lemma}\label{lm:conv_F+}
When $\ep\to0$ we have
\[\|\hat\cF_\ep-\cF\|_{\cL(L^1_\phi)}\to0.\]
\end{lemma}

\begin{proof}
Denoting $\rho_\ep=\int_\ep^1 z\wp(dz)$ we have
\begin{align*}
(\hat\cF_\ep&-\mathcal F)g(x)=\rho_\ep^{-1}\int_\ep^1\bigl(B+\eta\indic_{[A_\eta,+\infty)}\bigr)\Bigl(\frac xz\Bigr)g\Bigl(\frac xz\Bigr)\frac{\wp(\d z)}{z}-\eta\indic_{[A_\eta,+\infty)}(x)g(x)\\
&\hspace{75mm}-\int_0^1B\Bigl(\frac xz\Bigr)g\Bigl(\frac xz\Bigr)\frac{\wp(\d z)}{z}\\
&=(\rho_\ep^{-1}-1)\int_\ep^1\bigl(B+\eta\indic_{[A_\eta,+\infty)}\bigr)\Bigl(\frac xz\Bigr)g\Bigl(\frac xz\Bigr)\frac{\wp(\d z)}{z}-\int_0^\ep B\Bigl(\frac xz\Bigr)g\Bigl(\frac xz\Bigr)\frac{\wp(\d z)}{z}\\
&\hspace{38mm}+\eta\int_\ep^{\min(1,x/A_\eta)}g\Bigl(\frac xz\Bigr)\frac{\wp(\d z)}{z}-\eta\indic_{[A_\eta,+\infty)}(x)g(x).
\end{align*}
Mimicking the proof of Lemma~\ref{lm:F+def} we easily get
\begin{align*}
\|(\hat\cF_\ep&-\mathcal F)g\|\leq(\rho_\ep^{-1}-1)\biggl\|\int_0^1(B+\eta)\Bigl(\frac xz\Bigr)g\Bigl(\frac xz\Bigr)\frac{\wp(\d z)}{z}\biggr\|
+\biggl\|\int_0^\ep B\Bigl(\frac xz\Bigr)g\Bigl(\frac xz\Bigr)\frac{\wp(\d z)}{z}\biggr\|\\
&\hspace{74mm}+\eta\biggl\|\int_0^1g\Bigl(\frac xz\Bigr)\frac{\wp(\d z)}{z}\biggr\|+\eta\|g\|\\
&\leq(\rho_\ep^{-1}-1)C^2(\|B\|_\infty +\eta)\wp_0\|g\|+C^2\|B\|_\infty \biggl(\int_0^\ep\wp(\d z)\biggr)\|g\|+\eta C^2\wp_0\|g\|+\eta\|g\|
\end{align*}
which gives
\[\|\hat\cF_\ep-\mathcal F\|_{\cL(L^1_\phi)}\leq(\rho_\ep^{-1}-1)C^2(\|B\|_\infty +\eta)\wp_0+C^2\|B\|_\infty \biggl(\int_0^\ep\wp(\d z)\biggr)+\eta\Bigl(C^2\wp_0+1\Bigr).\]
Recall from Proposition~\ref{prop:eta_Aeta} that $\eta\to0$ when $\ep\to0.$
Additionally we have by monotone convergence $\rho_\ep\to1$ and $\int_0^\ep\wp(\d z)\to0.$
We conclude that $\hat\cF_\ep\to\mathcal F$ in $\cL(L^1_\phi).$
\end{proof}

We deduce the convergence of $\hat T_t^\ep$ to $T_t$ by using the Duhamel formula.

\begin{proof}[Proof of Proposition~\ref{prop:conv_Tt}]
For $g\in L^1_\phi$ we define $h_\ep=(\hat T_t^\ep-T_t)g.$
Clearly $h_\ep$ is the solution to
\[\partial_th_\ep=\cA h_\ep+(\hat\cA_\ep-\cA)\hat T_t^\ep g\]
with initial data $h_\ep(0)=0.$
Notice that $\hat\cA_\ep-\cA=-\hat\lambda_\ep+\lambda+\hat\cF_\ep-\cF$ is a bounded operator on $L^1_\phi.$
The Duhamel formula allows to write
\[h_\ep(t)=\int_0^tT_s(\hat\cA_\ep-\cA)\hat T_{t-s}^\ep g\,\d s\]
and we deduce that
\begin{align*}
\|\hat T_t^\ep-T_t\|_{\cL(L^1_\phi)}&\leq \|\hat\cA_\ep-\cA\|_{\cL(L^1_\phi)}\int_0^t\| T_s\|_{\cL(L^1_\phi)}\|\hat T_{t-s}^\ep\|_{\cL(L^1_\phi)}\,\d s\\
&\leq\|\hat\cA_\ep-\cA\|\int_0^t\| T_s\|\bigl(\|\hat T_{t-s}^\ep-T_{t-s}\|+\|T_{t-s}\|\bigr)\,\d s.
\end{align*}
Since $(T_t)$ is a contraction semigroup in $L^1_\phi$ we have $\| T_t\|_{\cL(L^1_\phi)}\leq1$ and
\[\|\hat T_t^\ep-T_t\|_{\cL(L^1_\phi)}\leq\|\hat\cA_\ep-\cA\|_{\cL(L^1_\phi)}\int_0^t\bigl(\|\hat T_{t-s}^\ep-T_{t-s}\|_{\cL(L^1_\phi)}+1\bigr)\,\d s.\]
By Gr\"onwall's lemma we get
\[\|\hat T_t^\ep-T_t\|_{\cL(L^1_\phi)}\leq \exp\bigl(t\|\hat\cA_\ep-\cA\|_{\cL(L^1_\phi)}\bigr)-1.\]
The conclusion follows from Lemma~\ref{lm:conv_lb} and Lemma~\ref{lm:conv_F+} since
\[\|\hat\cA_\ep-\cA\|_{\cL(L^1_\phi)}\leq|\hat\lambda_\ep-\lambda|+\|\hat\cF_\ep-\cF\|_{\cL(L^1_\phi)}.\]
\end{proof}

%
%
\section{Complementary results}\label{sec:complementary}
\subsection{Bound on the spectral gap in smaller spaces}
Here we consider fragmentation kernels such that $z_0>0$ and we quantify the maximum size of the spectral gap we can hope in smaller weighted $L^1$ spaces.
More precisely we consider spaces $L^1_\psi$ with weights $\psi$ which satisfy
\begin{equation}\label{as:psi}\exists c,C,r>0,\qquad\forall x\geq0,\qquad c\phi(x)\leq\psi(x)\leq C(1+x)^r\phi(x).\end{equation}
Define the {\it decay rate}
\[\omega_\psi:=\sup\bigl\{w\in\bR\,|\, \exists M\geq1,\, \forall t\geq0,\, \|T_t-P\|_{\cL(L^1_\psi)}\leq M\e^{-wt}\bigr\}.\]
Because of the conservativeness of the semigroup $(T_t)_{t\geq0}$ in $L^1_\phi$ we have that $\omega_\phi\geq0,$
and our main result Theorem~\ref{th:main} ensures that $\omega_\phi=0.$
Our proof strongly uses the fact that we work in the space $L^1_\phi$ and we cannot exclude the existence of a spectral gap for stronger weights $\psi.$
Actually it is proved in~\cite{MS} that there exists a spectral gap for strong enough weights, even for bounded fragmentation rates $B.$
But our method can be used to bound this spectral gap.

\begin{theorem}\label{th:bound_spectralgap}
Assume that {\bf(H$B$-H$\tau$-H$\wp$)} are satisfied and that $\wp$ is such that $z_0>0.$
For a weight $\psi$ which satisfies~\eqref{as:psi} we have
\[\omega_\psi\leq-\,\e\log z_0\|\cF_+\|_{\cL(L^1_\phi)}\,r.\]
\end{theorem}

\begin{proof}
Similarly as in the proof of Theorem~\ref{th:norm=2} (in which we take $X=1$) we write for $g_a=\frac{1}{\phi(x)}\indic_{a\leq x\leq a+1},$
\[\|T_tg_a-Pg_a\|_{L^1_\psi}\geq\int_0^1G(x)\psi(x)\,\d x-\int_0^1T_tg_a(x)\psi(x)\,\d x\]
and we have
\[\int_0^1T_tg_a(,x)\psi(x)\,\d x\leq2C\int_0^1T_tg_a(x)\phi(x)\,\d x\leq2CR_t(a).\]
Considering $a=\e^{\frac{w}{r} t}$ for some $w\in\bR$ and letting $t$ tend to $+\infty$ we get by using the Stirling formula
\begin{align*}
R_t(a)&=\sum_{n\geq\frac{w t}{-r\log z_0}}\frac{\|\cF_+\|_{\cL(L^1_\phi)}^n t^n}{n!}\\
&=\underset{t\to+\infty}{O}\bigg(\sum_{n\geq\frac{w t}{-r\log z_0}}\frac{(\|\cF_+\| t\e)^n}{n^n}\frac{1}{\sqrt{2\pi n}}\bigg)\\
&=\underset{t\to+\infty}{O}\bigg(\sum_{n\geq\frac{w t}{-r\log z_0}}\Bigl(\frac{-\e\log z_0\|\cF_+\|r}{w}\Bigr)^n\bigg)
\end{align*}
and the last term is the reminder of a convergent series provided that $w>-\,\e\log z_0\|\cF_+\|r.$
We deduce that there exists a constant $\delta>0,$ independent of $t,$ such that
\[\|T_tg_a-Pg_a\|_{L^1_\psi}\geq\delta\]
when $w>-\,\e\log z_0\|\cF_+\|r.$
To conclude we estimate
\[\|g_a\|_{L^1_\psi}=\int_a^{a+1}\frac{\psi(x)}{\phi(x)}\,\d x\leq C\int_a^{a+1}(1+x)^r\,\d x=\underset{a\to+\infty}{O}\big(a^r\big)=\underset{t\to+\infty}{O}\big(\e^{w t}\big).\]
\end{proof}

\noindent{\bf Example.}
In the case of the mitosis kernel $\wp=2\delta_{1/2}$ with constant fragmentation rate $B(x)\equiv B,$ we have $\wp_0=2,$ $z_0=\frac12,$ $\lambda=B$ and $\phi(x)\equiv1.$
From Lemma~\ref{lm:F+def} we get $\|\cF_+\|_{\cL(L^1_\phi)}\leq2B.$
Consider the weight $\psi(x)=(1+x)^r.$\\
From Theorem~\ref{th:bound_spectralgap} we deduce that $\omega_\psi\leq2B\e(\log2)r.$
We also know from~\cite{LP09} that $-B$ is an eigenvalue of the growth-fragmentation operator so that $\omega_\psi\leq2B.$\\
On the other hand it is proved in~\cite[proof of Proposition~6.5]{MS} that there exists a spectral gap for $r>\log_23$ and more precisely that $\omega_\psi\geq\max{(0,(2-3*2^{1-r})B)}.$\\
Finally we obtain the estimates
\[2B\max(0,1-3*2^{-r})\leq\omega_\psi\leq2B\min(\e(\log2)r,1).\]

\subsection{The homogeneous fragmentation kernel}
In the case of the homogeneous fragmentation kernel $\wp\equiv2$ we can say more about the asymptotic behavior of $\phi(x)$ at infinity,
under a stronger assumption on $\tau$ than~\eqref{as:tau_infty} but a weaker assumption on $B$ than~\eqref{as:B_infty}.

\begin{theorem}
Assume that $\wp\equiv2$ and that when $x\to+\infty,$ $B(x)\sim B_\infty x^\gamma$ and $\tau(x)\sim\tau_\infty x^\alpha$ with $\alpha<\gamma+1.$
Then there exists $\phi_\infty>0$ such that
\[\phi(x)\sim\phi_\infty x^k\qquad\text{when}\quad x\to+\infty\]
with
\[\begin{array}{ll}
k=1&\text{if}\quad \gamma>0,
\vspace{2mm}\\
\dis k=\frac{B_\infty-\lambda}{B_\infty+\lambda}\in(-1,1)\quad&\text{if}\quad \gamma=0,
\vspace{2mm}\\
k=\gamma-1&\text{if}\quad \gamma<0.
\end{array}\]
\end{theorem}

\begin{remark}
Notice that in the case $\gamma=0$ the result is consistent with Theorem~\ref{th:phi_estimates}.
Indeed $k=\frac{B_\infty-\lambda}{B_\infty+\lambda}$ is the power for which the relation $\wp_k=1+\frac{\lambda}{B_\infty}$ is satisfied since for $\wp\equiv2$ we have $\wp_k=\frac{2}{k+1}.$
\end{remark}

To our knowledge it is the first result which provides an equivalent of $\phi(x)$ when $x\to+\infty.$
In~\cite{BCG} we only have an upper and a lower bound with the same power law but not an equivalent.

\begin{proof}
We start from the equation for $\phi$
\[\tau(x)\phi'(x)=(\lambda+B(x))\phi(x)-2\frac{B(x)}{x}\int_0^x\phi(y)\,\d y\]
and we test this equation against $G$ on $[0,X]$
\[\int_0^X\tau(x)\phi'(x)G(x)\,\d x=\int_0^X(\lambda+B(x))\phi(x)G(x)\,\d x-2\int_0^X\frac{B(x)}{x}G(x)\int_0^x\phi(y)\,\d y\,\d x\]
to obtain, using $\tau(0)\phi(0)G(0)=0$ (see~\cite{DG10}),
\begin{align*}
\tau(X)\phi(X)G(X)&-\int_0^X\phi(x)(\tau G)'(x)\,\d x=\\
&\int_0^X(\lambda+B(x))G(x)\phi(x)\,\d x-2\int_0^X\phi(y)\int_y^X\frac{B(x)}{x}G(x)\,\d x\,\d y.
\end{align*}
Using the equation satisfied by $G$ we get
\[\tau(X)\phi(X)G(X)=2\int_0^X\phi(y)\int_X^\infty\frac{B(x)}{x}G(x)\,\d x\,\d y\]
which can be written as
\[\phi(x)=f(x)\int_0^x\phi(y)\,\d y\]
where we have defined
\[f(x):=\frac{2}{\tau(x)G(x)}\int_x^\infty\frac{B(y)}{y}G(y)\,\d y.\]
Denote by $F$ the function defined by
\[F(x):=\e^{-\int_1^xf(y)\,dy}\int_0^x\phi(y)\,\d y.\]
This function satisfies $F'(x)=0$ so we have for any $x>0,$ $F(x)=F(1)=\int_0^1\phi$ which gives
\[\int_0^x\phi(y)\,\d y=\Bigl(\int_0^1\phi\Bigr)\,\e^{\,\int_1^xf(y)\,\d y}\]
and then
\[\phi(x)=\Bigl(\int_0^1\phi\Bigr)\,f(x)\,\e^{\,\int_1^xf(y)\,\d y}.\]
Using again the equation satisfied by $G$ we can write
\[f(x)=\bigl(\log(\tau G)\bigr)'(x)+\frac{\lambda+B(x)}{\tau(x)}.\]
By integration we obtain
\[\int_1^x f(y)\,\d y=\log(\tau G)(x)-\log(\tau G)(1)+\Lambda(x)\]
and then
\[\e^{\,\int_1^xf(y)\,\d y}=\frac{\tau(x) G(x)}{\tau(1)G(1)}\e^{\Lambda(x)}.\]
Finally we get
\begin{equation}\label{eq:phi_wp=1}
\phi(x)=\Bigl(\int_0^1\phi\Bigr)\,\frac{2\,\e^{\Lambda(x)}}{\tau(1)G(1)}\int_x^\infty\frac{B(y)}{y}G(y)\,\d y.
\end{equation}
We know from~\cite{BCG} that there exists $C>0$ such that
\[G(x)\sim Cx^{\xi-\alpha}\e^{-\Lambda(x)},\]
where
\[\xi=\left\{\begin{array}{ll}
0&\quad\text{if}\ \gamma<0,
\vspace{2mm}\\
\frac{2B_\infty}{\lambda+B_\infty}&\quad\text{if}\ \gamma=0,
\vspace{2mm}\\
2&\quad\text{if}\ \gamma>0.
\end{array}\right.\]
Using the L'H\^opital's rule we get
\[\int_x^\infty\frac{B(y)}{y}G(y)\,\d y\sim \left\{\begin{array}{ll}
C\tau_\infty\frac{B_\infty}{\lambda}\,x^{\xi+\gamma-1}\e^{-\Lambda(x)}&\quad\text{if}\ \gamma<0,
\vspace{2mm}\\
C\tau_\infty\frac{B_\infty}{\lambda+B_\infty}\,x^{\xi-1}\e^{-\Lambda(x)}&\quad\text{if}\ \gamma=0,
\vspace{2mm}\\
C\tau_\infty\,x^{\xi-1}\e^{-\Lambda(x)}&\quad\text{if}\ \gamma>0.
\end{array}\right.\]
Finally by Equation~\eqref{eq:phi_wp=1} we obtain the existence of $\phi_\infty>0$ such that when $x\to+\infty$
\[\phi(x)\sim\phi_\infty x^k,\]
where the power $k$ is given by
\[k=\left\{\begin{array}{ll}
\gamma-1&\quad\text{if}\ \gamma<0,
\vspace{2mm}\\
\frac{B_\infty-\lambda}{B_\infty+\lambda}&\quad\text{if}\ \gamma=0,
\vspace{2mm}\\
1&\quad\text{if}\ \gamma>0.
\end{array}\right.\]

\end{proof}

\begin{corollary}
For the uniform fragmentation kernel $\wp\equiv2,$ our main result is still valid if we replace~\eqref{as:B_infty} by the more general condition
\[\lim_{x\to+\infty}B(x)=B_\infty<+\infty\]
while replacing~\eqref{as:tau_infty} by the stronger condition
\[\exists\, \alpha<1,\, \tau_\infty>0,\quad \lim_{x\to+\infty}x^{-\alpha}\tau(x)=\tau_\infty.\]
\end{corollary}

%
%
\section{Conclusion and perspectives}

We have proved that uniform exponential convergence does not hold in $L^1_\phi$ when the fragmentation rate is bounded.
The weight $\phi(x),$ which is the dual Perron eigenfunction of the growth-fragmentation operator, behaves as a power $k<1$ of $x$ at infinity.
We know from~\cite{MS} that an exponential rate of decay exists in smaller weighted $L^1$ spaces.
More precisely it exists in $L^1_\psi$ for $\psi(x)=(1+x)^r$ with $r>1$ large enough.
A natural question is to know what is the critical exponent for which we have uniform exponential convergence.
For instance does it occur for any $r>k$?

Another natural question concerns the unbounded case.
In the proof of Theorem~\ref{th:main} we crucially use the fact that the fragmentation rate is bounded.
This condition allows to control the number of splittings uniformly with respect to arbitrarily large initial sizes.
What happens when the fragmentation rate tends to infinity at infinity?
Do we have an exponential rate of decay in $L^1_\phi$ or not?

All these question will be addressed in future works.

\section*{Acknowledgements}

P.G. was supported by the ANR project KIBORD, ANR-13-BS01-0004, funded by the French Ministry of Research.

%
%


\end{document}